\documentclass[english]{article}\usepackage[T1]{fontenc}
\usepackage[latin9]{inputenc}
\usepackage{geometry}
\usepackage{color}
\usepackage{amsmath}
\usepackage{amssymb}

\makeatletter

\usepackage{amsthm}

\numberwithin{equation}{section}
\numberwithin{figure}{section}
\theoremstyle{plain}

\theoremstyle{plain}

\newtheorem{corollary}{Corollary}

\newtheorem{proposition}{Proposition}

\newtheorem{theorem}{Theorem}

\numberwithin{equation}{section}

\begin{document}
	\title
	{Special vector fields on Riemannian manifolds of constant negative sectional curvature and conservation laws}
	\author{ Keti Tenenblat$^{1}$ and  Alice Barbora Tumpach$^2$}
\date{}

\footnotetext [1] {Universidade de Bras\'\i lia - UnB, Departamento de Matem\'atica, 70910-900, Bras\'\i lia--DF, Brazil, K. Tenenblat was partially suported by CNPq  Grant; 311916/2021-0 and CAPES/Brazil- Finance Code 001; FAPDF 00193-00001678/2024-39;  ESI, Austria, grant  I-5015N , e-mail: {K.Tenenblat@mat.unb.br}    }
	\footnotetext [2] {Institut CNRS Pauli, Vienna-Austria;  Laboratoire Painlev\'e, Lille, France,   was partially supported by  CAPES-Print, Brazil, grant DGP/UnB 001/2024, e-mail: {Barbara.Tumpach@math.cnrs.fr}  }
	
\maketitle 	

\begin{abstract}
We show that any $n$-dimensional Riemannian manifold with constant negative sectional curvature 
admits local orthonormal vector fields such that one of them $v_1$ is tangent to geodesics and the other $n-1$ vector fields are tangent to horocycles. We prove that the $1$-form	 dual to $v_1$ is a closed form. We show how the closed form can be used to obtain conservation laws for PDEs whose generic solutions define metrics on open subsets with constant negative  sectional curvature. These results extend to higher dimensions the $2$-dimensional case proved in the 1980s. We prove that there exist local coordinates on the manifold such that the coordinate curves are tangent to the orthonormal vector  fields. We apply the theory to obtain conservation laws for the Camassa-Holm equation ($n=2$) and for the Intrinsic Generalized Sine-Gordon equation ($n\geq 2$). 
		\end{abstract}
	\textit{Keywords:} {Riemannian manifolds; constant negative sectional curvature; closed forms, conservation laws, Camassa-Holm equation; Intrinsic Generalized sine-Gordon equation }\\ 
	\textit{Mathematics Subject Classification}: [2020] {35A25, 35L65, 58A15, 58J60}

	\section{Introduction} 
	
In \cite{CavalcanteTenenblat1988},\cite{ChernTenenblat}, the authors proved that any 2-dimensional Riemannian manifold, with constant negative Gaussian curvature, admits  orthonormal vector fields $v_1,\, v_2$ tangent to geodesics and horocycles respectively.  In particular, they showed that the 1-form dual to $v_1$ is a closed form. The importance of the closed form is due to the fact that it  provides conservation laws for partial differential equations (or system of equations) for real valued functions, whose generic solutions define metrics on open subset of the plane, whose Gaussian curvature is constant negative (w.l.o.g. $-1$). These are the so called differential equations   that describe pseudo-spherical surfaces ({\bf pss}). Many well know differential equations related to physical phenomena  decribe {\bf pss} such as Schr\"odinger equation, short-pulse equation, KdV, etc.  Actually there are infinitely many such differential equations. The reader can find an extensive literature with classification results of such PDEs in \cite{ChernTenenblat}-\cite{SilvaTenenblat2015} and references within. Explicit conservation laws have been recently obtained for example in \cite{KelmerTenenblat2022} and \cite{KelmerTenenblat2025} for some of these equations, by applying the results in \cite{CavalcanteTenenblat1988}, 

In this paper, we generalize the results mentioned above to higher dimensions. More precisely, we show that any $n$-dimensional Riemannian manifold $M^n$, whose sectional curvature is constant $-1$, admits local orthonormal vector fields $v_i$, $i=1,...,n$, such $v_1$ is tangent to  geodesics and $v_i$, $i\geq 2$ 
are tangent to horocycles. In particular, we show that the 1-form dual to $v_1$ is a closed form (Theorem \ref{Theotheta1iij}).  Moreover, we prove that there exist local coordinates on $M$ such that the coordinate curves are tangent to the vectors of the orthonormal frame (Theorem \ref{coorddimn}). 
In Theorem \ref{closedformL},  we  show how to obtain the closed form for any Riemannian  
manifold, with constant sectional curvature $-1$.
This closed form provides conservation laws for PDEs whose generic solutions define metrics on open subsets of $R^n$ with constant negative sectional curvature. 
 We apply the results in dimension $n=2$ to obtain  conservation laws for the Camassa-Holm equation \cite{Cam-Holm}. For arbitrary dimensions $n\geq 2$, in Theorem \ref{conservationIGSGE} we get conservation laws for the Intrinsic Generalized Sine-Gordon equation (IGSGE). This is an $n$-dimensional generalization of the classical sine-Gordon equation, whose generic solutions define metrics on open subsets of $R^n$, whose sectional curvature is $-1$ (see Example 2). The IGSGE was introduced by Beals-Tenenblat in \cite{BealsTenenblat1991} (see also Chapter  V in \cite{Tenenblat}), as an intrinsic version of the Generalized sine-Gordon equation that corresponds to submanifolds $M^n\subset R^{2n-1}$ \cite{TT}.   
 In higher dimensions,  very few equations or systems of equations are known to be integrable in some sense.  The IGSGE  is an $n$-dimensional system of  PDEs that has B\"acklund transformation, superposition formula, and it can also be solved  by the inverse scattering method \cite{BealsTenenblat1991}. 

This paper is organized as follows: In Section \ref{Mainresults} we state our main results, in Section \ref{proofs} we prove Theorems \ref{Theotheta1iij}, \ref{coorddimn}  and \ref{closedformL}. In Section
\ref{applications}, we   obtain conservation laws for the Camassa-Holm equation and  in the higher dimensional context,  we prove Proposition \ref{conservedn} 
that shows how to apply Theorems \ref{Theotheta1iij} and \ref{closedformL}  in order to obtain conservation laws  from the closed 1-form and then  we  prove  Theorem \ref{conservationIGSGE} for the IGSGE.

	\section{Main Results}\label{Mainresults}
	
 We consider an $n$-dimensional Riemannian manifold $(M^n,g)$, with constant negative sectional curvature which, without loss of generality, we may consider to be $-1$. We first recall the  $2$-dimensional case, which shows that a Riemannian manifold  $(M^2,g)$, with constant negative Gaussian 
 curvature admits  special vector fields, that are tangent to geodesics and horocycles.
 
 \vspace{.1in}

	\begin{theorem} {\rm \cite{CavalcanteTenenblat1988}  \cite{ChernTenenblat}} \label{closedform}
Let $M^2$ be a $C^\infty$ Riemannian surface. $M$ has constant Gaussian curvature $-1$ if, and only if, given orthonormal vectors $v_1^0,\, v_2^0$ tangent to $M$ at $p_0\in M$, there exists an orthonormal frame field 
$v_1,v_2$, locally defined, such that $v_i(p_0)=v_i^0, \, i=1,2$ and the associated 
dual 1-forms $\theta_1\, 
\theta_2$ and connection form $\theta_{12}$ satisfy 
\begin{equation} \label{theta12}
\theta_{12}+\theta_2=0.
\end{equation}  
In this case, $\theta_1$ is a closed form.	
	\end{theorem}

\vspace{.2in}

We observe that the vector fields $v_1$ and $v_2$ are tangent to geodesics and to horocycles respectively. In fact, it follows from the fact that $dv_1=-\theta_2 v_2$ and $dv_2=\theta_2 v_1$.  We prove a higher dimensional version of the theorem above, whenever the Riemannian manifold has constant negative sectional curvature.

\begin{theorem}\label{Theotheta1iij}
Let $(M^n,g)$ be an $n$-dimensional Riemannian manifold. $M$ has constant sectional curvature $-1$ if, and only if, given $v_1^0, ...,v_n^0$ orthonormal vectors tangent to $M$ at $p_0\in M$, there exists an orthonormal 
frame field $v_1, ..., v_n$, locally defined, such that $v_i(p_0)=v_i^0, \, i=1,...,n$ and the associated dual forms $\theta_1, ...,\theta_n$ and connection form $\theta_{ij}$ satisfy
\begin{eqnarray}
&& \theta_{1i}+\theta_i=0,\qquad \forall\, i\geq 2, \label{theta1i}\\ 
&& \theta_{ij}=0, \qquad\qquad  \forall \, i\neq j, \;  2\leq i,j\leq n.\label{thetaij}
\end{eqnarray}
In this case, $\theta_1$ is a closed form. In particular, $v_1$ is tangent to geodesics and $v_i$, $i\geq 2$ are tangent to horocycles. 
\end{theorem} 

The existence of  special frames on a Riemannian manifold with constant negative sectional curvature, as in Theorem \ref{Theotheta1iij}, will enable us to show that one can locally parametrize the manifold with coordinates whose tangent vectors are in the direction of the frame.

\begin{theorem}\label{coorddimn}
Let $(M^n,g)$ be a Riemannian manifold of constant sectional curvature $-1$. Let $v_1,..., v_n$ be an orthonormal fame field locally defined on $M$ such that the dual forms $\theta_1,...,\theta_n$ and the connection forms $\theta_{ij}$  satisfy \eqref{theta12} if $n=2$ and  
\eqref{theta1i}-\eqref{thetaij} if $n>2$. Then there exist local coordinates $y_1,...,y_n$ such that $\partial/\partial y_1=v_1$ and $\partial/\partial y_i=r_iv_i$, $i\geq 2$, where $r_i$  is a funcion of $y_1$ only.
\end{theorem}

As an immediate corollary of Theorems \ref{Theotheta1iij} and \ref{coorddimn} we have 

\begin{corollary}\label{coordinates}
Let $(M^n,g)$ be a Riemannian manifold of constant sectional curvature $-1$.
Then given orthonormal vectors   
 $v_1^0,..., v_n^0$  at a point $p^0\in M$,  there exist local orthogonal coordinates $y_1,...,y_n$ such that the curves that are  tangent to $\partial/\partial y_1$ are geodesics and  the curves tangent to $\partial/\partial y_i$, $i\geq 2$, are horocycles and at $p^0$ they are tangent to $v_1^0,..., v_n^0$. 
 \end{corollary}

Whenever $n=2$,  Theorem \ref{closedform}  has been applied to obtain conservation laws  for differential equations that describe pseudo-spherical surfaces, i.e.,  differential equations or system of equations for real valued functions defined on open subsets of the plane, whose generic solutions define metrics with Gaussian curvature $-1$. 

In order to consider such an application for higher dimensions, i.e., for systems of differential equations whose generic solutions define metric on open subsets of $R^n$ with sectional curvature $-1$, we state our next result,  that shows how to obtain the closed form given by Theorem \ref{Theotheta1iij}.

\begin{theorem}\label{closedformL}
Let $(U\subset R^n,g)$ be an open subset $U$ with  a Riemannian metric $g$,  
whose sectional curvature is constant $-1$ and let $e_1,...,e_n$ be any orthonormal frame 
with dual and connection forms $\omega_i$ and $\omega_{ij}$ respectively. Then there exists a unique orthonormal frame $v_i=L_{ij} e_j$, $L(x)\in O(n)$, for a given initial condition $L(x^0)\in O(n)$, $x^0\in U$,  such that $L$ satisfies the integrable system of PDEs  
\begin{eqnarray}
&& (dL\, L^t)_{1i}+(LWL^t)_{1i}+\sum_{k=1}^n L_{ik}\omega_k=0,\qquad \forall i\geq 2, \label{dL1j}\\
&& (dL\, L^t)_{ij}+(LWL^t)_{ij}=0, \qquad \qquad \forall\, i, j \geq 2, \, i\neq j, \label{dLij}
\end{eqnarray}
where $(W)_{ij}=\omega_{ij}$. 
In this case, $\sum_{k=1}^n L_{1k} \omega_k$ is a closed form.
\end{theorem} 

The existence of a closed form in higher dimensions provides conservation laws, as one can see in Proposition  \ref{conservedn}, when  we fix one of the independent variables to be the time variable. As an important application, we  consider the  Intrinsic Generalized sine-Gordon equation, which is a system of differential equations for  a pair $\{V(x),h(x) \}$ defined on an open set $x\in U\subset R^n$, $n\geq 2$, where $V(x)=(V_1(x),...,V_n(x))$ is a unit vector field and $h_{ij}(x)$ is an off diagonal $n\times n$ matrix valued function satisfying \eqref{IGSGE} (see Example 2). This equation reduces to the sine-Gordon equation when $n=2$.  By applying  Theorem \ref{closedformL}, we prove the following result.

\begin{theorem}\label{conservationIGSGE} The Intrinsic Generalized sine-Gordon equation admits 
at least $n-1$ conservation laws, considering one of the independent variables to be the time variable. 
\end{theorem}

 \section{Proof of the Main Results}\label{proofs}
In this section, we prove some of the results stated in Section \ref{Mainresults}. In order to do so, 
we need the following basic facts. Let $(M^n,g)$ be a Riemannian manifold of constant sectional curvature $K$.  Consider a local orthonormal frame field $e_1,...,e_n$. Let $\omega_1,...,\omega_n$ be its dual coframe and 
let $\omega_{ij}=-\omega_{ji}$ be the connection forms. Then the structure equations for $M$ 
are
\begin{eqnarray}
& & d\omega_i=\sum_{j\neq i, \,j=1}^n\omega_j\wedge \omega_{ji}, \label{domegai}\\ 
& & d\omega_{ij}=\sum_{k=1}^n\omega_{ik}\wedge \omega_{kj} +  \Omega_{ij}
\label{domegaij} 
\end{eqnarray} 
where the curvature  $\Omega_{ij}=-K \omega_i\wedge \omega_j$  characterizes the fact that the sectional curvature is constant $K$.

\vspace{.1in} 

In what follows,   
we will use the notion of {\em vector valued differential forms}  on a manifold.  In particular, 
let $v:M^n\rightarrow TM$ be a vector field on $M$, where $TM$ is the tangent bundle of $M$. One can consider $v_p=\sum_{i=1}^n v^i(p) \, e_i(p)$, where $e_i(p)$ is a basis of the tangent space $T_pM$. Then $dv:TM\rightarrow TM$ is a vector valued differential form given by $dv(X)=\sum_{i=1}^n d v^i(X) \, e_i,$  where $dv^i$ is a 1-form and $X$ is any tangent vector field on $M$.  Just as in the case of ordinary differential forms, one can define operations on vector valued forms such as addition, multiplication by a function, wedge product and exterior derivatives  acting component-wise relative to any basis of the vector space. In particular, denoting the metric by $<,>$,  whenever two vector fields $v_i$ and $v_j$ are such that $<v_i,v_j>$ is constant, then 
$
<dv_i(v_k),v_j>+<v_i,dv_j(v_k)>=0, 
$
for any vector field $v_k$.

\vspace{.1in}


\begin{proof}[\bf Proof of Theorem \ref{Theotheta1iij}]
We have  to prove that the system of equations \eqref{domegai} and \eqref{domegaij} is integrable 
if, and only if, the sectional curvature $K$  of $M$ is constant $-1$. In order to do so, we will use Cartan-K\"ahler theory on exterior differentiable systems \cite{Cartan}. 

Let $\cal{I}$ be the ideal generated by $\gamma_i=\theta_{1i}+\theta_i$ and $\beta_{ij}=\theta_{ij}$, $i\neq j$, $i,j\geq 2$.  Then, it follows from \eqref{domegai} and \eqref{domegaij}, that  
\begin{eqnarray} 
&& d\gamma_i=d\theta_{1i}+d\theta_i= d\theta_{1i}+\sum_{k=1}^n\theta_k\wedge \theta_{ki}\nonumber\\
&& \qquad = d\theta_{1i}+\theta_1\wedge \theta_{1i}+\sum_{k=2}^n\theta_k\wedge \theta_{ki}\nonumber\\
&& \qquad = d\theta_{1i}+\theta_1\wedge(\gamma_i-\theta_i)+\sum_{k=2}^n(\gamma_k-\theta_{1k})\wedge \theta_{ki}\nonumber\\
&& \qquad = d\theta_{1i}-\sum_{k=2}^n \theta_{1k}\wedge \theta_{ki}-\theta_1\wedge \theta_i \quad \mbox{( mod $\cal{I}$)}.\label{lastdgammai} 
\end{eqnarray}
Similarly, for $i,j\geq 2$
\begin{eqnarray} 
&& d\beta_{ij}=d\theta_{ij}=\sum_{k=1}^n\theta_{ik}\wedge \theta_{kj}+\Omega_{ij}\nonumber\\
&&\qquad = \theta_{i1}\wedge \theta_{1j}+\sum_{k=2}^n\theta_{ik}\wedge \theta_{kj}+\Omega_{ij}\nonumber\\
&& \qquad =-(\gamma_i-\theta_i)\wedge (\gamma_{j}-\theta_{j})+\sum_{k=2}^n\beta_{ik}\wedge \beta_{kj}+\Omega_{ij}\nonumber \\
&&\qquad = - \theta_i\wedge\theta_j +\Omega_{ij} \quad \mbox{( mod $\cal I$)}.\label{lastdbetaij}  
 \end{eqnarray}
Therefore, $d\beta_{ij}=0$ mod $\cal{I}$ whenever $\Omega_{ij} =\theta_i\wedge \theta_j $, 
for all $i,j\geq 2$.  Hence, it follows from \eqref{lastdgammai} and \eqref{lastdbetaij} that 
$\cal{I}$ is closed under exterior differentiation if and only if $M$ has constant sectional curvature $K=-1$. The first part of the theorem follows from Frobenius theorem. 
  
  We now prove that $\theta_1$ is a closed form. In fact, it  follows from \eqref{theta1i} and the structure equation \eqref{domegai} that
  \[
  d\theta_1=-\sum_{k=2}^n\theta_k\wedge \theta_{1k}=\sum_{k=2}^n\theta_k\wedge \theta_{k}=0.
  \]
  
  We observe  that as a consequence of \eqref{theta1i} and \eqref{thetaij} we have  $dv_1=-\sum_{i=2}^n\theta_i v_i$ and  $dv_i=\theta_i v_1$. Hence, $dv_1(v_1)=0$ and $dv_i(v_i)=v_1$ Therefore, the vector fields $v_1,...,v_n$ have the following property:  $v_1$ is tangent to  geodesics and   $v_i$, $i\geq 2$ are tangent to horocycles.

\end{proof}

Our next proof shows that manifolds of constant negative sectional curvature admit 
special coordinate systems locally defined as in Theorem \ref{coorddimn}.

\vspace{.2in}

\begin{proof}[\bf Proof of Theorem \ref{coorddimn}]
We start proving the 2-dimensional case, i.e., let $M^2$ be a Riemannian surface. 
 Let $v_1,\, v_2$ be an orthonormal frame field locally defined on $M$ such that the dual forms $\theta_1,\, \theta_2$ and the connection form $\theta_{12}$ satisfy \eqref{theta12}. 
 It follows from Theorem \ref{closedform} that $\theta_{1}$ is a closed form. Therefore, there exists a function $G$ locally defined such that 
 \begin{equation}\label{G}
 dG= \theta_1.
\end{equation}  
We need to show that there exits a function $r$ such that  $v_2(r)=0$ and the vector fields $v_1$ and 
$rv_2$ commute, i.e.,  
\begin{equation}\label{ry1}
dr(v_2)=0 \qquad \mbox{ and } \qquad [v_1,\, rv_2]=0. 
\end{equation}

Since $\theta_{12}=<dv_1,v_2>$ it follows from \eqref{theta12} that 
\begin{equation}\label{dv1}
\begin{array}{l}
<dv_1(v_1),v_2>\;=\theta_{12}(v_1)\;=-\theta_2(v_1)=\;0,\\
<dv_1(v_2),v_2>\;=\theta_{12}(v_2)\;=-\theta_2(v_2)=\,-1.\\
\end{array}
\end{equation}
The second equality of \eqref{ry1} is equivalent to   
\[
dr(v_1)v_2+rdv_2(v_1)-rdv_1(v_2)=0.
\]
This is a vector field that vanishes when its inner product with the basis vanishes.  Taking  inner products of this expression with $v_1$ and $v_2$, we get from \eqref{dv1} respectively
\begin{equation}\label{dlogr}
\begin{array}{l}
   r <dv_2(v_1),v_1>-r<dv_1(v_2),v_1>= -r< v_2, dv_1(v_1)>=0, \\ 
dr(v_1)-r<dv_1(v_2),v_2>=dr(v_1)+r=0.
\end{array}
\end{equation}
Therefore, $d (\log r)(v_1)=-1$. Moreover, from the first equality of  \eqref{ry1} we have $d (\log r)(v_2)=0$.
Since $d (\log r)= d (\log r)(v_1)\theta_1+d (\log r)(v_2)\theta_2$, we conclude from 
\eqref{dlogr} that $d(\log r)=-\theta_1$. It follows from \eqref{G} that  
\[
r=c\exp(-G),
\] 
for some constant $c>0$, i.e., the function $r$ exists satisfying \eqref{ry1}. Therefore, there exit coordinates $y_1$ and $y_2$ locally defined such that  $\partial/\partial y_1=v_1$ and $\partial/\partial y_2=rv_2$, where $r$  is a funcion of $y_1$ only.

\vspace{.1in}

We  now prove the $n$-dimensional case. 
Let $v_1,..., v_n$ be an orthonormal frame field locally defined on $(M^n,g) $ such that the dual forms $\theta_1,...,\theta_n$ and the connection forms $\theta_{ij}$ satisfy \eqref{theta1i} and \eqref{thetaij}. 
 It follows from Theorem \ref{Theotheta1iij} that $\theta_{1}$ is a closed form. Therefore, there exists a function $G$ locally defined such that 
 \begin{equation}\label{Gn}
 dG= \theta_1.
\end{equation}  
We need to show that there exit  functions $r_j\neq 0$, $j\geq 2$,  such  $v_i(r_j)=0$, for $i,j\geq 2$  and the vector fields $v_1,r_2v_2,...,r_n v_n$ commute, i.e.,  
\begin{equation}\label{ry1n}
dr_i(v_j)=0 \qquad [v_1,\, r_iv_i]=0 \quad \mbox{ and } \quad [r_iv_i,r_jv_j]=0, \quad \forall i,j\geq2. 
\end{equation}

Since 
\begin{equation}\label{theta1iijP}
\theta_{1i}=<dv_1,v_i> \mbox{ and } \quad \theta_{ij}=<dv_i,v_j>,
\end{equation}
 it follows from \eqref{theta1iijP}, \eqref{theta1i}
and \eqref{thetaij}  that for all $i\neq j$, $i,j\geq 2$ we have 
\begin{equation}\label{dv1dvi}
\begin{array}{l}
<dv_1(v_1),v_i>\;=\theta_{1i}(v_1)\;=-\theta_i(v_1)=\;0,\\
<dv_1(v_i),v_i>\;=\theta_{1i}(v_i)\;=-\theta_i(v_i)=\,-1,\\
<dv_1(v_i),v_j>\;=\theta_{1j}(v_i)\;=-\theta_i(v_j)=\,0,\\
<dv_i(v_k),v_j>\;=\theta_{ij}(v_k)\;=\,0, \; \forall k.
\end{array}
\end{equation}

The second and third equalities of  \eqref{ry1n} are equivalent to 
\begin{eqnarray}
&& dr_i(v_1)v_i+r_idv_i(v_1)-r_idv_1(v_i)=0, \label{commut1i}\\
&& r_idr_j(v_i)v_j+r_ir_j dv_j(v_i)-r_jdr_i(v_j)v_i-r_jr_idv_i(v_j)=0, \quad i,j\geq 2.\label{commutij}
\end{eqnarray}
The expressions given by \eqref{commut1i} and \eqref{commutij}are  vector fields that vanish 
whenever the inner product with all vectors $v_1,...,v_n$ vanish. 
 
Taking inner product of \eqref{commut1i} with $v_1$ and $v_j, \; j\geq 2$, 
 from \eqref{dv1dvi} and \eqref{theta1iijP} and the fact that $<dv_1(v_i),v_1>=0$, we get that 
\begin{equation*}\label{dlogrn}
\left\{ \begin{array}{l}
r_i<dv_i(v_1),v_1>= - r_i<dv_1(v_1),v_i>=0, \\
dr_i(v_1)\delta_{ij} +r_i\delta_{ij}=0, 
\end{array}\right.
\end{equation*}
respectively.  Therefore, $[v_1,\, r_iv_i]=0$ if and only if 
\begin{equation}\label{dlogriv1}
d(\log r_i)(v_1)=-1.
\end{equation}

Similarly, taking inner product of \eqref{commutij} with $v_1$ and $v_k,\; k\geq 2$, we get from \eqref{theta1iijP} and the third equation of \eqref{dv1dvi} that, for all $i,j\geq 2$, 
\begin{equation*}\label{drivk}
\left\{\begin{array}{l}
r_ir_j\left(\,<dv_j(v_i),v_1>-<dv_i(v_j),v_1>\,\right)= r_ir_j\left(-<dv_1(v_i),v_j>+ <dv_1(v_j),v_i>\,\right)=0, \\
r_idr_j(v_i)\delta_{jk}-r_jdr_i(v_j)\delta_{ik}=0. 
\end{array}\right.
\end{equation*}
Therefore, $dr_k(v_i)=0$ for  $i\neq k, \, i,k\geq 2$. Now, since we also want the first equality of \eqref{ry1n} to be satisfied, i.e.,  $dr_k(v_k)=0$,  we conclude that   $[r_iv_i,r_jv_j]=0$ whenever  $dr_k(v_i)=0$  for all $i,k\geq 2$.
Since $d (\log r_i)= d (\log r_i)(v_1)\theta_1+\sum_{j=2}^n d (\log r_i)(v_j)\theta_j$, we conclude from 
\eqref{dlogriv1}   that $d(\log r_i)=-\theta_1$.     Therefore,    it follows  from \eqref{Gn} that  
\[
r_i=c_i\exp(-G),
\] 
for some constant $c_i>0$, i.e., the functions $r_i$ exist satisfying \eqref{ry1n}. Therefore, there exist coordinates $y_1,... y_n$ locally defined such that  $\partial/\partial y_1=v_1$ and $\partial/\partial y_i=r_iv_i
$, where $r_i$  is a funcion of $y_1$ only.

\end{proof}

 In order to prove Theorem \ref{closedformL} we will  use  Theorem \ref{Theotheta1iij} that says that there are special vector fields on  a Riemannian manifold of constant negative sectional curvature. Therefore we need the  following basic result that shows how the dual and connection forms are affected under a change of orthonormal frame fields, on any Riemannian manifold.

\begin{proposition}\label{novoframe}
Let $(U\subset R^n,g)$ be an open subset $U$ with  a Riemannian metric $g$,   
and $e_1,...,e_n$ is  an orthonormal frame on $U$ with dual forms $\omega_1,...,\omega_n$ and connection forms $\omega_{ij}=-\omega_{ji}$. If $v_1,...,v_n$ is another orthonormal frame given by $v_i=\sum _{j=1}^n L_{ij}e_j$ where  $L(x)\in O(n)$
$x\in U$, then its dual forms $\theta_i$ and its connection forms $\theta_{ij}$ are given by 
\begin{equation}\label{thetaieij}
\theta_i=\sum _{j=1}^n L_{ij}\omega_j \quad \mbox{ and }\quad 
\theta_{ij}=(dL\, L^t)_{ij}+(LWL^t)_{ij}, 
\end{equation}
 where $L^t$ is the transpose of $L$ and $(W)_{ij}=\omega_{ij}$.
\end{proposition}  

The proof is a straightforward computation  showing that $\theta_i(v_\ell)=\delta_{i\ell}$ and $\theta_{ij}=<dv_i,v_j>$.

\vspace{.2in}

\begin{proof}[\bf Proof of Theorem \ref{closedformL}]  
The proof of Theorem \ref{closedformL} follows from Theorem \ref{Theotheta1iij} and Proposition \ref{novoframe}.  

\end{proof}

\section{Applications}\label{applications}

In this section, we provide some applications of the results stated in Section \ref{Mainresults}. 
\subsection{The two dimensional case}\label{apln2}
Theorem \ref{closedform} has been applied in order to obtain conservation laws for several PDEs for a real function $u(x,t)$ (or systems of PDEs) that describe pseudo-spherical surfaces (see \cite{CavalcanteTenenblat1988}, \cite{ChernTenenblat}). Such an equation is  characterized by the fact that its generic solutions define metrics, on open subsets of the plane,  
whose Gaussian curvature is constant $-1$, i.e., they define 1-forms $\omega_1$ and $\omega_2$ and 
the connection forms $\omega_{12}$ in terms of $u(x,t)$ and its derivatives which satisfy the 
structure equations
\begin{equation}\label{structure}
\begin{array}{l}
d\omega_1=\omega_2\wedge \omega_{21}\\
d\omega_2=\omega_1\wedge \omega_{12}\\ 
d\omega_{12}=\omega_1\wedge \omega_{2}
\end{array}
\end{equation}
 The metric is defined by $ds^2=\omega_1^2+\omega_2^2$.

 The conservation laws for the differential equations are  obtained 
 as a consequence of Theorem \ref{closedform} as follows. Consider the 1-forms $\omega_i$, $i=1,2$ and $\omega_{12}$ satisfying \eqref{structuren}. Then there are 
 orthonormal vector fields $e_1$ and $e_2$ whose dual forms are $\omega_1$ and $\omega_2$. 
 Moreover, Theorem \ref{closedform} says that there exist special  orthonormal frames $v_1$ and $v_2$ such that its dual forms $\theta_1$, $\theta_2$ and   its connection form 
 $\theta_{12}$ satisfy
 $\theta_{12}+\theta_2=0$, and in this case $\theta_1$ is a closed form.

 Observe that the frames $e_1,\; e_2$ and $v_1,\; v_2$ are related  by an angle function $\phi$,  namely
 \[
 \begin{array}{l}
 v_1=\cos \phi\, e_1-\sin\phi\, e_2,\\
 v_2=\sin \phi\, e_1+\cos\phi\, e_2, 
 \end{array}
 \]
 and hence the dual forms and the connection forms are related by 
 \begin{equation} \label{phin2}
 \begin{array}{l}
 \theta_1=\cos \phi\, \omega_1-\sin\phi\, \omega_2,\\
 \theta_2=\sin \phi\, \omega_1+\cos\phi \,\omega_2, \\
 \theta_{12}=\omega_{12}- d\phi.
 \end{array}
 \end{equation} 
Therefore,  it follows from \eqref{theta12} that $\phi$ is determined up to constants by the following equation 
\begin{equation}\label{dphin2} 
d\phi =\omega_{12} + \sin \phi\, \omega_1+\cos\phi \,\omega_2, 
\end{equation}
 and   the closed 1-form $\theta_1$ is given  in terms of $\phi$ by \eqref{phin2}, i.e., 
 \begin{equation}\label{closedn2}
\theta_1= \cos \phi\, \omega_1-\sin\phi\, \omega_2.
 \end{equation}

Assume  that, in a coordinate system  $(x,t)\in U\subset R^2$, the 1-forms 
are given as follows
\[
\omega_i=f_{i1}\,dx+f_{i2}\, dt \qquad \omega_{12}=f_{31}\,dx+f_{32}\, dt,  
\]
where $i=1,2$ and $f_{ij}(x,t)$ are differentiable functions. Then the angle function $\phi(x,t)$ is determined in terms of $f_{ij} $ by \eqref{dphin2},  i.e.,
  by the completely integrable  system of equations  
\begin{equation}\label{eqsolphi}
 \begin{cases}
\phi_x=f_{31}+f_{11} \sin\phi +f_{21} \cos\phi,\\
\phi_t=f_{32}+f_{12} \sin\phi +f_{22} \cos\phi.
\end{cases}
\end{equation}  
Moreover,  if $\phi(x,t)$ is any solution of \eqref{eqsolphi}, then  \eqref{closedn2} 
implies that 
\begin{equation}\label{closedn2fij}
(f_{11} \cos \phi-f_{21} \sin \phi)\, dx +(f_{12} \cos \phi-f_{22} \sin \phi)\, dt 
\end{equation}
is a closed form that provides a conservation law.  

Assume that moreover the functions $f_{ij}$ are analytic in a parameter $\eta$, then the angle  function  $\phi$  and the closed form  \eqref{closedn2fij} will also be analytic in $ \eta$. In this case, we consider $\phi=\sum_{j=0}^\infty \phi_j \eta^j$. 
Then the  Laurent expansion of \eqref{eqsolphi}  will provide the functions $\phi_j,\, j\geq 0$ and the expansion of \eqref{closedn2fij} will provide  infinitely many 
 closed forms, by considering each coefficient of the powers of $\eta$. 
 When one considers  solutions that are periodic on $x$ 
 or solutions that decay to zero when $x\rightarrow \pm \infty$, then one gets infinitely many conserved quantities.

 We mention recent results in \cite{KelmerTenenblat2022}  and  \cite{KelmerTenenblat2025}, where the conservation laws were 
  explicitly  given for a Pholmeyer-Lund Regge type system of equations and for a vector Short pulse equation.  In this paper, we apply  this procedure in detail 
 to the Camassa-Holm equation.
 
 \vspace{.2in}

 \noindent{\bf Example 1.}
Consider the Camassa-Holm equation    
\begin{equation}\label{eqCH}
u_t-u_{xxt}=uu_{xxx}+2u_xu_{xx}-3uu_x-mu_x. 
\end{equation} 
 whose generic solutions define metrics on open sets of the 
 plane $(x,t)$, whose Gaussian curvature is constant $-1$. In fact, in \cite{SilvaTenenblat2015}
 it was shown that considering the 1-forms
 \begin{eqnarray*}
&& \omega_i=f_{i1}dx+f_{i2}dt, \qquad i=1,2, \\ 
&& \omega_{12}=f_{31}dx+f_{32}dt,
 \end{eqnarray*}
where 
\[
\begin{array}{ll}
 f_{11}=h-1+\frac{\eta^2}{2}\qquad & f_{12}=-u(f_{11}+1)+\eta u_x-\frac{m}{2}-\frac{\eta^2}{2}+1, \\
f_{21}=\eta & f_{22}=u\eta+u_x-\eta, \\
f_{31}=h+\frac{\eta^2}{2}& f_{32}= \, -u f_{31} +\eta u_x-u-\frac{m}{2} -\frac{\eta^2}{2}, 
 \end{array}
\]
and $h(x,t)=u-u_{xx}+m/2$, the metric defined by $ds^2=\omega_1^2+\omega_2^2$, where $\omega_{12}$ is the connection form,  has Gaussian curvature -1 . More precisely, the structure equations \eqref{structure} are satisfied if, and only if,  the function $u(x,t)$ satisfies 
\eqref{eqCH}. 

Since the functions $f_{ij}$ are analytic in $\eta$, it follows that the angle function $\phi$  satisfying  \eqref{eqsolphi} is also  analytic in $\eta$ and we can  consider $\phi=\sum_{j=0}^\infty \phi_j \eta^j$. Therefore, the closed 1-form  \eqref{closedn2fij} is also analytic in $\eta$. 
By considering each coefficient of $\eta$ in \eqref{eqsolphi},  we get the integrable system of differential equations that provide the functions $\phi_j,\, j\geq 0$ and the  Laurent expansion of \eqref{closedn2fij} provides infinitely many closed forms considering the coefficients of the powers of $\eta$. 

More precisely, from the Laurent expansion of \eqref{eqsolphi}, considering the coefficient independent of $\eta$,  the following system is integrable for $\phi_0$ (i.e.  the mixed derivatives commute),  for any solution $u$ of \eqref{eqCH}. 
\[
\left\{ 
\begin{array}{l}
\phi_{0,x}=(h-1)\sin \phi_0+h,\\
\phi_{0,t}=\cos \phi_0 u_x-\sin\phi_0 (uh+\frac{m}{2}-1)-u(h+1)-\frac{m}{2}.
\end{array}
   \right.
\]
Therefore, there exists a unique $\phi_0$ for a given initial condition $\phi_0(x_0,t_0)$.
Moreover, from \eqref{closedn2fij} we have a closed form for \eqref{eqCH}   given by 
\begin{equation}\label{firstCH}
 \cos\phi_0 (h-1)\, dx +[ \cos\phi_0(-uh+1-m/2)-u_x\sin\phi_0]\, dt. 
\end{equation}

Similarly, considering $u$ and $\phi_0$ as above,  the following system is integrable for $\phi_1$  
\[
\left\{ 
\begin{array}{l}
\phi_{1,x}=[(h-1)\phi_1+1]\cos \phi_0,\\
\phi_{1,t}=\,-[ (uh-1+m/2)\cos\phi_0+u_x\sin\phi_0]\phi_1-(u+1)\cos\phi_0+u_x(\sin\phi_0+1), 
\end{array}
   \right.
\]
and we have a second closed form for \eqref{eqCH}  given by 
\begin{equation}\label{secondCH}
-\{ [(h-1)\phi_1+1]\sin\phi_0 \}\, dx 
+\{ u_x(\phi_1+1)\cos\phi_0+[(uh+m/2-1)\phi_1+u+1]\sin\phi_0\} \, dt.
\end{equation}

We observe that this procedure goes on for all coefficients of $\eta$ in the Laurent expansion,  providing infinitely many closed forms for the Camassa Holm equation 
 and hence conserved quantities in time when the functions are periodic in $x$ or decay appropriately  when $x\rightarrow \infty$.

\subsection{Higher dimensional case}
We  will now show  how Theorems \ref{Theotheta1iij} and \ref{closedformL}  can be applied in the higher dimensional context. We consider a system of PDEs for functions with $n$ independent 
 variables $(x_1,...,x_{n})$, $n\geq 2$,  whose generic solutions define  Riemannian metrics  on open sets 
 of $R^n$ such that the sectional curvature is constant $-1$. This means that we have 1-forms $\omega_1,...,\omega_n$ and connection 1-forms $\omega_{ij} =-\omega_{ji}$, $1\leq i,j\leq n$ given in terms 
 of the solutions of the PDEs and its derivatives, such that the following structure of equations are satisfied
\begin{equation}\label{structuren}
\begin{array}{l}
  d\omega_i=\sum_{j\neq i, \,j=1}^n\omega_j\wedge \omega_{ji}, \\ \\ 
  d\omega_{ij}=\sum_{k=1}^n\omega_{ik}\wedge \omega_{kj} + \omega_i\wedge \omega_j. 
\end{array}
\end{equation}
 The metric is defined by $ds^2=\sum_{i}^n\omega_i^2$, where $\omega_i$ $i=1,...n$, are linearly independent. 
 
 \vspace{.1in}
 
The following result shows how to obtain conservation laws from a closed 1-form in 
higher dimensions.
 
\vspace{.2in}

\begin{proposition}\label{conservedn}
Consider local coordinates $x=(x_1, x_2,...,x_n)\in U\subset R^n$, $n>2$ and  a 1-form  
 $
\theta=\sum_{j=1}^n f_{j}(x) dx_j,  
$    
 where $f_{j}$ are differentiable functions of $x$.  Assume that $x_1$ is the time variable that is denoted by $t$. If $\theta$ is a closed form 
 then  for each $j\geq 2$, the $(n-1)-$forms
 \begin{equation}\label{closedn-1}
 \begin{array}{l}
  \psi_2=\theta \wedge dx_3 \wedge dx_4 \wedge \ldots \wedge dx_n,\\  
 \psi_3=\theta\wedge dx_2\wedge dx_4\wedge\ldots \wedge dx_n,\\
\vdots \\ 
 \psi_n=\theta\wedge dx_2 \wedge \ldots \wedge dx_{n-1}
 \end{array}
 \end{equation}
 are conservation laws and hence 
\begin{equation}
\frac{\partial f_{1}}{\partial x_j}-\frac{\partial f_{j}}{\partial t}=0, \qquad \forall j\geq 2.  
\label{indepjt}
\end{equation} 
  Moreover, if the functions $f_j$ are analytic in  a parameter $\eta$, then the closed form may provide infinitely many conservation laws. 
\end{proposition}

\begin{proof}
Since $\theta= f_{1}(x)\, dt+ \sum_{\ell=2}^n f_{\ell}(x) dx_\ell,$ is a closed form, it follows that 
for all $j\geq 2$, $\psi_j$ is also a closed form. 
 \begin{eqnarray*}
d\psi_j&=&d\theta\wedge d x_2\wedge\ldots \wedge dx_{j-1}\wedge dx_{j+1} \wedge\ldots dx_n\\
  &=&
 \sum_{\ell=2}^n\left( \frac{\partial f_1}{\partial x_\ell}dx_\ell\wedge dt+
\frac{\partial f_j}{\partial x_\ell }dx_\ell\wedge dx_j \right)
\wedge d x_2\wedge\ldots \wedge dx_{j-1}\wedge dx_{j+1} \wedge\ldots dx_n\\
&=& \left( \frac{\partial f_j}{\partial t} -\frac{\partial f_1}{\partial x_j}\right) 
dt\wedge dx_2\wedge\ldots\wedge dx_n.
\end{eqnarray*}
Since  $d\psi_j=0$ and $dt,\, dx_2, \ldots dx_n$ are linearly independent  we get 
that \eqref{indepjt} holds. 

 If the functions $f_j$ are analytic in a parameter $\eta$, then each coefficient of the Taylor expansion of $\theta$ in terms of $\eta$, provides a closed form and hence a conservation law. 

\end{proof}

{\bf Remarks:}
\begin{enumerate}
\item Proposition \ref{conservedn} shows that for each $j\geq 2$, $\int f_{j} dx_j$ is a conserved quantity in time, for functions $f_{1}$  
that decay appropriately when $x_i\rightarrow \infty$ for $i=2,...,s$  for some $s\leq n-1$  and 
$f_1$ is periodic on the remaining variables $x_j$, $j=s+1,...,n$, i.e., $f_1$  is defined on 
  $R\times R^s \times T^{n-s-1} $.    
In fact, it follows from \eqref{indepjt} that 
\[
\frac{\partial }{\partial t}\int f_{j}dx_j= \int \frac{\partial f_{1} }{\partial x_j}dx_j=0, 
\] 
i.e. \eqref{indepjt} provides $n-1$ conserved quantities.
\item The closed form $\theta$ of Proposition \ref{conservedn},  besides implying  \eqref{indepjt}, 
it also shows that for $i,j\geq 2$ and $i<j$,  the functions $f_i$ and  $f_j$ satisfy the following relations
\[
\frac{\partial f_{j}}{\partial x_i}-\frac{\partial f_{i}}{\partial x_j}=0.  
\]

\end{enumerate}

\vspace{.2in}

As an  application of Theorem \ref{closedformL}, we will show how to obtain  conservation laws for the Intrinsic Generalized sine-Gordon (the sine-Gordon equation when $n=2$).  

\vspace{.2in}

\noindent {\bf Example 2.} 
The {\em Intrinsic Generalized sine-Gordon } (IGSGE) is a  system of second order differential equations for a unit vector field  $V(x)=(V_1(x),V_2(x),...,V_n(x))$,  $x\in U\subset R^n$ that can be regarded as a first order system of equations for the pair $\{ V, h \} $, where $h$ is an off-diagonal   $(n\times n)$-matrix valued function determined  by the first derivatives of $V$ (when $V_i$ do not vanish), see the second equation in \eqref{IGSGE}), 
given by 
\begin{eqnarray}
&& VV^t=1, \nonumber\\
&& \frac{\partial V_i}{\partial x_j}=V_j h_{ji}, \nonumber \\
&& \frac{\partial  h_{ij}}{\partial x_i } + \frac{\partial h_{ji} }{\partial x_j }
 +\sum_{s\neq i,j} h_{si}h_{sj}=V_iV_j, \qquad  i\neq j, \label{IGSGE}  \\
&&  \frac{\partial h_{ij}}{\partial x_s }=h_{is} h_{sj}, \qquad i,j,s \mbox{ distinct}. \nonumber 
\end{eqnarray} 
where $1\leq i,j,s\leq n$.  When $n=2$, by taking $V=(\cos u/2,\, \sin u/2)$, $u(x_1,\, x_2)$, 
the system \eqref{IGSGE} reduces to the sine-Gordon equation
\[
u_{x_1x_1}-u_{x_2x_2}=\sin u. 
\]
Explicit solutions of the IGSGE are given for example by 
\[
V_1=\tanh x_1, \qquad V_j=c_j \frac{1}{\cosh x_1}, \qquad \mbox{ where } x_1>0, \mbox{ and } 
\sum_{j=2}^n c_j^2=1.
\]
Other solutions for \eqref{IGSGE} can be found in \cite{Tenenblat} (Example b) page 142 and Proposition 3.1 on page 143).

\vspace{.2in}

We observe that considering a pair $\{V,h\}$ satisfying \eqref{IGSGE} such that $V_i$ do not vanish on an open subset $U\subset R^n$, then the unit vector field $V$ defines a Riemannian metric on $U$ with constant sectional curvature $-1$.  In fact, considering the one forms 
\begin{equation}\label{metric}
 \omega_i=V_i dx_i,\qquad 
 \omega_{ij}=h_{ij}dx_j-h_{ji}dx_i, 
\end{equation}
the metric is defined by $ds^2=\sum_{i=1}^n \omega_i^2$ and the 1-forms $\omega_i$ and the connection forms $\omega_{ij}$  satisfy the structure equations \eqref{structuren},  as a consequence of \eqref{IGSGE}. Observe that the third and fourth equations of the system  \eqref{IGSGE} are the Gauss equation of the metric.   

  
\begin{proof}[\bf Proof of Theorem \ref{conservationIGSGE} ]
As we have seen above, any solution $\{V,h\}$ of the IGSGE  defines a metric on an open subset $U\subset R^n$, whose sectional curvature is constant $-1$ by considering  the 1-forms 
$\omega_i$ and $\omega_{ij}$ defined by \eqref{metric}. 

Let $e_i, \;  i=1,...,n$ be the  orthonormal vector fields whose dual forms are $\omega_i$. 
It follows from Theorem \ref{closedformL}, that there exists a unique orthonormal frame  $v_i=L_{ij}e_j$,  $L(x)\in O(n)$, for a given initial condition $L(x^0)$, $x^0\in U$ such that for all $i,j\geq 2,\, i\neq j$, equations \eqref{dL1j} and \eqref{dLij} hold and in this case $\sum_{k=1}^n L_{1k}\omega_k$ is a closed form. Hence, considering 
\eqref{metric}, the system of equations  \eqref{dL1j} and \eqref{dLij} reduces to
\begin{eqnarray}
&& (dL\, L^t)_{1i}+(LWL^t)_{1i}+\sum_{k=1}^n L_{ik}V_k\, dx_k=0,\label{dL1jIGSGE}\\
&& (dL\, L^t)_{ij}+(LWL^t)_{ij}=0,  \label{dLijIGSGE}
\end{eqnarray}
where $i,j\geq 2,\, i\neq j$ and 
\begin{equation}\label{bigW}
W_{ij}=\omega_{ij}=h_{ij} dx_j- h_{ji} dx_i.
\end{equation}
In this case,   
\begin{equation}\label{closedformn}
\theta=\sum_{k=1}^n L_{1k}V_k dx_k
\end{equation}
 is a closed form. Hence,  it follows from Proposition 
\ref{conservedn} that, assuming  that $x_1=t$ is  a time variable, then for all $k\geq 2$, 
 the $(n-1)-$forms $\psi_k$ given by \eqref{closedn-1} are conservation laws.
and hence   
\begin{equation}\label{conservedquantity}
\frac{\partial (L_{1k}V_k)}{\partial t}-
\frac{\partial (L_{11}V_1)}{\partial x_k}=0. 
\end{equation}

\end{proof}

As we mentioned in Example 2, whenever $n=2$, the Intrinsic Generalized Sine-Gordon Equation reduces to the classical sine-Gordon equation $u_{x_1x_1}-u_{x_2x_2}=\sin u$, for a real valued function $u(x_1,x_2)$, when we consider  $V=(\cos (u/2),\sin (u/2))$. We want to exhibit the conservation law  given by Theorem \ref{conservationIGSGE} for the sine-Gordon equation.

In this case, for any solution $u(x_1,x_2)$ of the sine-Gordon equation, it follows from the second equation \eqref{IGSGE} that $h_{12}=u_{x_1}/2$ and 
$h_{21}=-u_{x_2}/2$. Therefore,  \eqref{metric} provides the metric and the connection form as   
\[
\omega_1=\cos\frac{u}{2}dx_1, \qquad \omega_2=\sin\frac{u}{2}, \qquad 
\omega_{12}=\frac{u_{x_1}}{2}dx_2+\frac{u_{x_2}}{2}dx_1.
\]
It follows from Theorem \ref{closedformL} that there exists an orthogonal matrix 
\[
L=\left( 
\begin{array}{cc}
\cos\phi & -\sin\phi\\
\sin\phi & \cos \phi
\end{array}
\right), 
\]
where $\phi(x_1,x_2)$, that satisfies \eqref{dLijIGSGE} and \eqref{dL1jIGSGE}.  
Then, \eqref{dLijIGSGE} cannot occur since $n=2$ and \eqref{dL1jIGSGE} occurs only for $i=2$ and it  reduces to 
\[
-d\phi+\omega_{12}+\sin\phi\cos\frac{u}{2}dx_1+\cos\phi\sin\frac{u}{2}dx_2 =0.
\]  
Since $d\phi=\phi_{x_1}dx_1+\phi_{x_2}dx_2$, it follows from the expression of $\omega_{12}$ that 
$\phi$ must satisfy
\begin{equation}\label{phix1phix2}
\phi_{x_1}=\frac{u_{x_2}}{2}+\sin\phi\cos\frac{u}{2} \qquad  \mbox{and}\qquad  
\phi_{x_2}=\frac{u_{x_1}}{2}+\cos\phi\sin\frac{u}{2}.
\end{equation}
This is an integrable system for $\phi$ since $u$ satisfies the sine-Gordon equation.
 For any solution  $\phi$ of \eqref{phix1phix2}, we get from \eqref{closedformn}  the closed form
\[
\theta=\cos\phi\cos\frac{u}{2} dx_1-\sin\phi\sin\frac{u}{2} dx_2.
\] 
By considering $x_1=t$ to be the time variable, we have the conservation law  
$\theta=\cos\phi\cos\frac{u}{2} dt-\sin\phi\sin\frac{u}{2} dx_2$.
Hence, $\int -\sin\phi\sin\frac{u}{2} dx_2$ is a conserved quantity whenever  the function $\cos\phi\cos(u/2)$  is either periodic in $x_2$ or it decays appropriately when $x_2$ tends to infinity.    

\vspace{.2in}

\noindent {\bf Remark.} We observe that  in Theorem \ref{conservationIGSGE},  whenever the metric and  the connection forms associated to the solutions of the IGSGE  
are   analytic in a parameter,  then \eqref{conservedquantity} may provide infinitely many conservation laws  in time. This occurs very often in the 2-dimensional case as one can see in the literature of differential equations that describe pseudo-spherical surfaces (see  \cite{ChernTenenblat}-\cite{SilvaTenenblat2015}).  
 
 \vspace{.1in}

\noindent{\bf Acknowlegments:} 
This research was partially supported by Fonds zur F\"orderung der wissenschaft\-lichen Forschung, Austria grant  I-5015N ``Banach Poisson--Lie groups and integrable systems''.
The authors would like to acknowledge the excellent working conditions and interactions at Erwin Schr\"odinger International Institute for Mathematics and Physics, Vienna, during the thematic programme "Infinite-dimensional Geometry: Theory and Applications" where part of this work was completed.
This research was also partially supported by the Brazilian Ministry of Education, grant DGP/UnB 001/2024 CAPES-Print, while the second author was visiting 
the Mathematics Department at the Universidade de Bras\'\i lia, Brazil, where this work was initiated.


\end{document}